\documentclass{amsart}[11pt]
\usepackage[utf8]{inputenc}
\usepackage{amsmath, amssymb, amsthm}
\usepackage[left=3cm, right=3cm, top=2cm, bottom=2cm]{geometry}
\usepackage{enumitem}
\usepackage[numbers,sort]{natbib}
\usepackage{thmtools, thm-restate}
\usepackage{tikz}
\usepackage{subcaption}
\usepackage{algorithm}
\usepackage[noend]{algpseudocode}

\usepackage{hyperref, color}

\DeclareMathOperator{\E}{\mathbb{E}}

\DeclareMathOperator{\Hom}{Hom}

\DeclareMathOperator{\im}{im}

\newtheorem{theorem}{Theorem}
\newtheorem{question}[theorem]{Question}
\newtheorem{corollary}[theorem]{Corollary}
\newtheorem{lemma}[theorem]{Lemma}
\newtheorem{claim}[theorem]{Claim}

\theoremstyle{definition}
\newtheorem{definition}[theorem]{Definition}

\theoremstyle{remark}
\newtheorem{remark}[theorem]{Remark}

\numberwithin{theorem}{section}
\numberwithin{definition}{section}
\numberwithin{lemma}{section}
\numberwithin{corollary}{section}
\numberwithin{remark}{section}
\numberwithin{claim}{section}

\title{Threshold Graphs Maximize Homomorphism Densities}
\author{Grigoriy Blekherman and Shyamal Patel}
\date{}

\begin{document}

\maketitle

\begin{abstract}
Given a fixed graph $H$ and a constant $c \in [0,1]$, we can ask what graphs $G$ with edge density $c$ asymptotically maximize the homomorphism density of $H$ in $G$. For all $H$ for which this problem has been solved, the maximum is always asymptotically attained on one of two kinds of graphs: the quasi-star or the quasi-clique. We show that for any $H$ the maximizing $G$ is asymptotically a threshold graph, while the quasi-clique and the quasi-star are the simplest threshold graphs, having only two parts. This result gives us a unified framework to derive a number of results on graph homomorphism maximization, some of which were also found quite recently and independently using several different approaches.  We show that there exist graphs $H$ and densities $c$ such that the optimizing graph $G$ is neither the quasi-star nor the quasi-clique \cite{day2019conjecture}. We also show that for $c$ large enough all graphs $H$ maximize on the quasi-clique \cite{gerbner2018maximum}, and for any $c \in [0,1]$ the density of $K_{1,2}$ is always maximized on either the quasi-star or the quasi-clique \cite{ahlswede1978graphs}. Finally, we extend our results to uniform hypergraphs.
\end{abstract}

\section{Introduction}
In this paper, we asymptotically study the number of homomorphisms from a fixed graph $H$ to graphs $G$ with a fixed edge density. Specifically, we investigate the properties of graphs $G$ that maximize the \emph{homomorphism density} from $H$ to $G$:

    \begin{definition}[Homomorphism Density]
        Denote the homomorphism density of $H$ in $G$ by 
            \[t(H,G) = \frac{\hom(H,G)}{|G|^{|H|}},\]
        where $|G|$ denotes the number of vertices in $G$ and $\hom(H,G)$ denotes the number of homomorphisms from $H$ to $G$.
    \end{definition}

Formally, for a given graph $H$ and $c \in [0,1]$, we are interested in finding a sequence of graphs that attains the value of

    \[\mathcal{M}_H(c) = \limsup_{G \in \mathcal{C}_c} t(H,G), \]

where $\mathcal{C}_c = \{G: t(K_2, G) \leq c\}$ and the $\limsup$ over a class of graphs is defined as 
    \[\limsup_{G \in \mathcal{C}} f(G) = \lim_{n \rightarrow \infty} \sup \{f(G): G \in \mathcal{C}, |G| \geq n\}.\]

This quantity has been studied for a variety of graphs $H$. Two families of graphs that frequently maximize $t(H, \cdot)$ are the quasi-clique and quasi-star, where a quasi-clique is an induced clique with isolated vertices and a quasi-star is the complement of a quasi-clique (see Figure \ref{fig:test}). We remark that in some papers, the quasi-clique is defined as isolated vertices and an induced clique with one additional vertex that may only be connected to a subset of vertices in the clique \cite{nagy2017number, day2019conjecture}. This is important if one wants to determine the non-asymptotic maximizer of $t(H,G)$ for a graph $G$ with say $n$ vertices and $m$ edges; however, we can use a simpler definition as we only consider asymptotic results.

\begin{figure}

\centering
\begin{subfigure}{.5\textwidth}
  \centering
	\begin{tikzpicture}[scale=1, transform shape]
	    \node[circle, inner sep=0pt, minimum size=2mm, fill] (A) at (0,.75) {};
	    \node[circle, inner sep=0pt, minimum size=2mm, fill] (B) at (.5,.1) {};
	    \node[circle, inner sep=0pt, minimum size=2mm, fill] (C) at (-.5,.1) {};
	    \node[circle, inner sep=0pt, minimum size=2mm, fill] (D) at (-.5,-.75) {};
	    \node[circle, inner sep=0pt, minimum size=2mm, fill] (E) at (.5,-.75) {};
	    
	    \draw (A) -- (B);
	    \draw (A) -- (C);
	    \draw (A) -- (D);
	    \draw (A) -- (E);
	    \draw (B) -- (C);
	    \draw (B) -- (D);
	    \draw (B) -- (E);
	    \draw (C) -- (D);
	    \draw (C) -- (E);
	    \draw (D) -- (E);
	    \draw (0,0) ellipse [x radius = 1cm, y radius = 1.15cm];
	    
	    \draw (3,0) ellipse [x radius = 1cm, y radius = 1.15cm];
	     
	    \node[circle, inner sep=0pt, minimum size=2mm, fill] (X) at (3,.75) {};
	    \node[circle, inner sep=0pt, minimum size=2mm, fill] (Y) at (3,0) {};
	    \node[circle, inner sep=0pt, minimum size=2mm, fill] (Z) at (3,-.75) {};
	    
	\end{tikzpicture}

    \caption{Quasi-clique}
  
\end{subfigure}%
\begin{subfigure}{.5\textwidth}
  \centering
	  \begin{tikzpicture}[scale=1, transform shape]
	    \node[circle, inner sep=0pt, minimum size=2mm, fill] (A) at (-.25,.6) {};
	    \node[circle, inner sep=0pt, minimum size=2mm, fill] (B) at (.5,0) {};
	    \node[circle, inner sep=0pt, minimum size=2mm, fill] (C) at (-.25,-.6) {};
	    
	    \draw (A) -- (B);
	    \draw (A) -- (C);
	    \draw (B) -- (C);
	    \draw (0,0) ellipse [x radius = 1cm, y radius = 1.15cm];
	    
	    \draw (3,0) ellipse [x radius = 1cm, y radius = 1.15cm];
	     
	    \node[circle, inner sep=0pt, minimum size=2mm, fill] (X) at (3,.75) {};
	    \node[circle, inner sep=0pt, minimum size=2mm, fill] (Y) at (3,0) {};
	    \node[circle, inner sep=0pt, minimum size=2mm, fill] (Z) at (3,-.75) {};
	    
	    \draw (A) -- (X);
	    \draw (A) -- (Y);
	    \draw (A) -- (Z);
	    \draw (B) -- (X);
	    \draw (B) -- (Y);
	    \draw (B) -- (Z);
	    \draw (C) -- (X);
	    \draw (C) -- (Y);
	    \draw (C) -- (Z);
	    
	\end{tikzpicture}
    \caption{Quasi-star}
\end{subfigure}
\caption{An example of a quasi-star and quasi-clique}
\label{fig:test}
\end{figure}
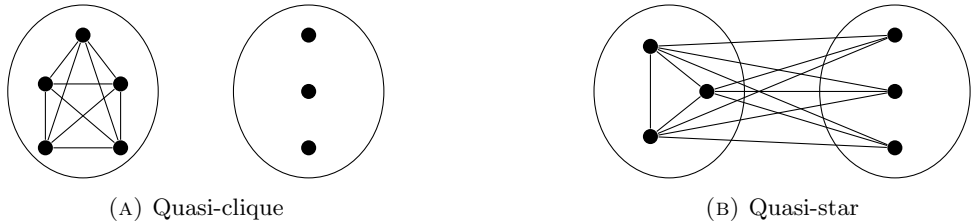

 A very general result of Alon implies that if $H$ has a spanning subgraph that is a disjoint union of cycles and isolated edges then $\mathcal{M}_H(c)$ is maximized on the quasi-clique for all $c$ \cite{alon1981number}. The study of the behavior of specific $H$ largely began with Ahlswede and Katona, who showed that for any $c \in [0,1]$ when $H$ is the $2$-star, a star with two edges, $\mathcal{M}_H(c)$ is always achieved on either the quasi-star or quasi-clique \cite{ahlswede1978graphs}. This result was later generalized to $k$-stars, showing that for any $c \in [0,1]$ the number of homomorphisms from the $k$-star is maximized when $G$ is the quasi-star or quasi-clique for small $k$ \cite{kenyon2017multipodal}, and shortly after for all $k \geq 2$ \cite{reiher2018maximum}. The question was also studied in the case where $H$ is the $4$-edge path, and again it was shown that the optimizing graph is always either the quasi-star and quasi-clique for all densities \cite{nagy2017number}.

A particularly important class of graphs for this problem are threshold graphs, which contain both the quasi-clique and quasi-star.

\begin{definition}(Threshold Graph)
\label{def:threshold-graph-nested}
A graph $T$ is threshold if for any two vertices $u,v \in V(T)$ we have that $N(u) \subseteq \overline{N(v)}$ or $N(v) \subseteq \overline{N(u)}$.
\end{definition}

We remark that there are other characterizations of threshold graphs \cite{mahadev1995threshold} which we will discuss in more detail in Section \ref{section:prelim}. Our main contribution is proving via a local move that the maximizer of $t(H,G)$ for any graph $H$ is always attained on a threshold graph.

    \begin{restatable}{theorem}{thresholdMaximizationTheorem}
        \label{thresholdMaximizationTheorem}
        For any graph $H$ and $c \in [0,1]$, we have that
            \[\mathcal{M}_H(c) = \limsup_{G \in \mathcal{C}_c} t(H,G) = \limsup_{G \in \mathcal{C}_c \cap \mathcal{T}} t(H,G), \]
        where $\mathcal{T}$ denotes the set of all threshold graphs. 
    \end{restatable}

Such a result is of interest as threshold graphs have simpler limit objects than general graphs. Instead of considering the $\limsup$ over all graphs, we can work directly with graphons, which are graph limit objects, and find the graphon that maximizes the number of homomorphisms from $H$ \cite{lovasz2012large}. This approach is employed in the results of Nagy \cite{nagy2017number} and Reiher and Wagner \cite{reiher2018maximum}. Threshold graphs may be more convenient to work with as their limits are one-dimensional, as opposed to graphons which are two dimensional \cite{diaconis2008threshold}.

We then extend Theorem \ref{thresholdMaximizationTheorem} to sparse graphs and hypergraphs. To do so, we define 
	\[\mathcal{M}(H,n,m) = \max \{\hom(H,G): |V(G)| \leq n, |E(G)| \leq m\}, \]
since homomorphism densities are zero for sparse graphs. Note that the function $\mathcal{M}(H,n,m)$ is well understood up to constant factors (depending on $H$) by a result of Janson et al. \cite{janson2004upper} (cf. Theorem \ref{subgraphCountTheorem}). 

    \begin{restatable}{theorem}{sparseThresholdMaximizationTheorem}
        \label{sparseThresholdMaximizationTheorem}
            For any graph $H$ and function $m(n) = \omega(n^{3/2})$, we have that
                \[\mathcal{M}(H,n,m) \sim \max \{\hom(H,T): |V(T)| \leq n, |E(T)| \leq m, T \in \mathcal{T}\}, \]
            where $\mathcal{T}$ denotes the set of threshold graphs. Moreover if $H$ has no induced $C_4$ or $P_4$, then we have equality for any positive integers $n$ and $m$ i.e.
	            \[\mathcal{M}(H,n,m) = \max \{\hom(H,T): |V(T)| \leq n, |E(T)| \leq m, T \in \mathcal{T}\}. \]
	\end{restatable}

We say that $f(n) \sim g(n)$ if $\lim_{n \rightarrow \infty} f(n)/g(n) = 1$. 

We can also prove a similar result for suitably defined threshold hypergraphs.

\begin{definition}[Threshold Hypergraph]
\label{def:threshold-hypergraph}
Let $G$ be a $k$-uniform hypergraph with $k \geq 2$. We then say that $G$ is a threshold hypergraph if for any two vertices $x,y$ we either have that $x \ll_G y$ or $y \ll_G x$, where we say $x \ll_G y$ if for any $e \in E(G)$ with $x \in e$ and $y \not \in e$ we have that $(e \setminus \{x\}) \cup y \in E(G)$. We omit the subscript when the underlying hypergraph $G$ is clear from context.
\end{definition}

We remark that there are several non-equivalent definitions of threshold hypergraphs \cite{reiterman1985threshold}, but we'll use Definition \ref{def:threshold-hypergraph} throughout this paper.

    \begin{restatable}{theorem}{hypergraphThresholdMaximizationTheorem}
    \label{hypergraphThresholdMaximizationTheorem}
        Let $H$ be a fixed $k$-uniform hypergraph and $c \in [0,1]$ then we have that 
         \[\mathcal{M}_H(c) = \limsup_{G \in \mathcal{C}_c} t(H,G) = \limsup_{G \in \mathcal{C}_c \cap \mathcal{T}} t(H,G). \]
            \end{restatable}
In a slight abuse of notation, we take $\mathcal{T}$ above to be the set of threshold hypergraphs and $\mathcal{C}_c$ to be the set of hypergraphs of edge density at most $c$.

The remainder of the paper shows a variety of applications of these ideas. We first give a simple proof in Theorem \ref{cliqueMaximizationTheorem} of a recent result by Gerbner et al. \cite{gerbner2018maximum} that every graph is maximized on the quasi-clique for $c$ sufficiently close to $1$. More precisely, we show that if $\mathcal{M}_{K_{1,|H|-1}}(c)$ is asymptotically maximized on the quasi-clique then so is $\mathcal{M}_{H}(c)$. 

We then contrast this with an example in Theorem \ref{threePartsTheorem} of a graph $H$ which for some $c\in [0,1]$ is optimized on neither the quasi-star nor the quasi-clique, disproving a conjecture of Nagy \cite{nagy2017number}. To do so, we explicitly find a graph $G$ that has more homomorphisms from $H$ than the quasi-star and quasi-clique when $c$ is small. This result was proved independently by Day and Sarkar \cite{day2019conjecture}, who used a similar argument. While the arguments are fundamentally the same, the class of graphs $\mathcal{G}$ that we use with more homomorphisms from $H$ than the quasi-star or quasi-clique is the next simplest threshold graph, fitting well into our results.

Finally, we reprove a result that the two-star is maximized on either the quasi-star or quasi-clique.
    \begin{restatable}{theorem}{cherryMaximizationTheorem}
    \label{cherryMaximizationTheorem}
        For any $c \in [0,1]$, we have that $\mathcal{M}_{K_{1,2}}(c)$ is asymptotically attained on the quasi-star or quasi-clique.
    \end{restatable}

\section{Preliminaries}
\label{section:prelim}
\subsection{Notation}
We typically use $H$ as the fixed graph, $G$ as the target graph, $c$ as the edge density of $G$, and let $n$ and $m$ denote the number of vertices and edges in $G$. We'll let $\Hom(H,G)$ denote the set of homomorphisms from $H$ to $G$ and $\hom(H,G)$ be the cardinality of $\Hom(H,G)$. Moreover, we will be consistent with the notation described in the introduction.

We denote the neighborhood of a vertex $u$ in $G$ by $N_G(u)$, and the closed neighborhood of a vertex $u$ by $\overline{N}_G(u) := N_G(u) \cup \{u\}$. We frequently drop the subscript $G$ when the graph is clear from context. 

Often times when using big-Oh notation, there will be implicit constants depending on the graph $H$ or other parameters $k$. In this case, we denote this using a subscript e.g. $O_H(f(n))$ or $O_k(f(n))$.

\subsection{Threshold Graphs}
We begin by recalling some basic facts about threshold graphs. There are a number of equivalent definitions of threshold graphs and for a thorough treatment of the subject we refer the reader to \cite{mahadev1995threshold}. A key second characterization of threshold graphs will be

\begin{definition}[Threshold Graph]
A graph is a threshold if it can be built, starting from a single vertex graph, by repeatedly adding dominating or isolated vertices.
\end{definition}

In light of the above definition, we put the set of $n$ vertex threshold graphs in one-to-one correspondence with binary sequences of length $n - 1$. In such a sequence the $i$th element is a $1$ if the $i$th vertex that we added is a dominating vertex and a $0$ if the $i$th vertex added is an isolated vertex.

Moreover, when using this representation we will often refer to the number of parts of a threshold graph. This refers to the number of blocks in the corresponding binary string of the given threshold graph. For instance, $1011100$ corresponds to a threshold graph with $4$ parts. Similarly, we note that the quasi-clique and quasi-star are threshold graphs with $2$ parts.

Finally, while we do not use this definition directly, we state it as it shows that threshold graphs satisfy the criteria of Corollary \ref{exactMaxCorollary}, below. 

\begin{definition}[Threshold Graph]
A graph is threshold if and only if it contains no induced copies of a $4$ vertex path, a $4$ vertex cycle, or the complement of a $4$ vertex cycle.
\end{definition}

\subsection{Basic Properties of $\mathcal{M}_H(c)$}
There are a number of basic properties of $\mathcal{M}_H(c)$ that we rely on. For instance, the function is continuous. We will say a homomorphism $\varphi: V(H) \rightarrow V(G)$ uses an edge $uv \in E(G)$ if there exists an edge $xy \in H$ such that $u = \varphi(x)$ and $v = \varphi(y)$. Similarly, it uses a vertex $v \in V(G)$ if $v \in \im \varphi$.

\begin{lemma}
\label{homcontinuitylemma}
For any fixed graph $H$, $\mathcal{M}_H(c)$ is continuous.
\end{lemma}

\begin{proof}
Let $\epsilon > 0$ and $c \in [0,1]$.  Note that there are at most $k n^{|H|-2}$ homomorphisms using a given edge, where $k$ is a constant only depending on $H$. Then it follows that there are at most $k \delta n^{H}$ homomorphisms using any set of at most $\delta n^2$ edges. We claim it suffices to take $\delta = \epsilon / k$.

To see this, let $c' < c$ such that $|c' - c| < \delta$. Now if $G$ is a graph with edge density $c$ such that $t(H,G) > \mathcal{M}_H(c) - \frac{\epsilon}{2}$, then we have that removing any subset of $(c - c')n^2/2$ edges gives us a graph $G'$ with
    \[t(H,G') \geq t(H,G) - \epsilon/2 \geq \mathcal{M}_H(c) - \epsilon.\]
Hence, it follows that $\mathcal{M}_H(c') \geq \mathcal{M}_H(c) - \epsilon$. Since $\mathcal{M}_H(c')$ is clearly non-decreasing, it follows that $|\mathcal{M}_H(c') - \mathcal{M}_H(c)| \leq \epsilon$.
\end{proof}

We now also note that while in the definition of $\mathcal{M}_H(c)$ we take a $\limsup$ we could have equivalently defined it as the supremum over all graphs in $\mathcal{C}_c$. This follows from the following useful lemma

\begin{lemma}
\label{homSeqLemma}
Let $G$ be a graph on $n$ vertices, then there exists a sequence of graphs $G = G_1, G_2, \hdots$ such that $|G_{i+1}| > |G_i|$ and $t(H,G) = t(H,G_{i})$ for any graph $H$. 
\end{lemma}

\begin{proof}
Let $G$ be a graph with $n$ vertices, $v_1, ..., v_n$. If $A$ denotes the adjacency matrix of $G$, then consider $G^2$, which we define to be the graph with adjacency matrix
    \[\left[
        \begin{array}{c|c}
          A & A \\
          \hline
          A & A
        \end{array}
    \right].\]

Now we claim that $\hom(H,G^2) = 2^{|H|} \hom(H,G)$. To see this we note that we can construct $G^2$ by adding vertices $v_1', ..., v_n'$, where we have edges $v_i v_j$, $v_iv_j'$, $v_i'v_j$, and $v_i' v_j'$ if and only if $v_i v_j \in E(G)$. Let $\pi: V(G^2) \rightarrow V(G)$ be the map that sends $v_i'$ to $v_i$ and sends $v_i$ to itself. Now suppose that $\varphi \in \hom(H,G^2)$. Consider $\pi \circ \varphi: H \rightarrow V(G)$ and note that by the definition of $G^2$ we have that $\pi \circ \varphi \in \hom(H,G)$. From this, we see that there are at most $2^{|H|} |\hom(H,G)|$ homomorphisms from $H$ to $G^2$. But conversely, we can see that given a homomorphism $\varphi_1 \in \hom(H,G)$, if $\varphi_2: V(H) \rightarrow V(G^2)$ is a function such that $\pi \circ \varphi_2 = \varphi_1$, then $\varphi_2$ is a homomorphism. Hence, we have that $\hom(H, G^2) = 2^{|H|} \hom(H,G)$ and thus that $t(H,G) = t(H,G^2)$. Taking $G_{i+1} = (G_i)^2$ then gives the desired result.
\end{proof}

Note that the sequence of graphs does not depend on $H$. As such, by taking $H = K_2$, we have that each graph $G_i$ in the sequence will also have the same edge density as $G$.

\subsection{Fractional Independence Number}
To prove that there exists a graph not optimized on the quasi-clique or quasi-star, we will need to define the fractional independence number, which is closely related to the number of homomorphisms by a result of Janson et al. (cf. Theorem \ref{subgraphCountTheorem}) \cite{janson2004upper}.

\begin{definition}[Fractional Independence Number]
\label{def:fractional-independence-number}
Given a graph, $G = (V, E)$, with vertices $v_1, ... ,v_n$, the fractional independence number denoted $\alpha^*(G)$ is

\begin{equation*}
\begin{aligned}
\max_{w_1, ... w_n} & \sum_i w_i \\
\textrm{s.t.} \; \; & w_i + w_j \leq 1, \;\; v_i v_j \in E, \\
\qquad & w_i \in [0,1].
\end{aligned}
\end{equation*}

\end{definition}

Since we can take all the $w_i = \frac{1}{2}$, we always have that $\alpha^*(G) \geq \frac{|G|}{2}$. It is well-known that the polytope defined by the constraints above is half-integral i.e. all its vertices are in $\{0, 1/2, 1\}^n$. Since we were not able to locate a good reference for this fact, we give a short proof below.

\begin{lemma}
	The fractional independence polytope is half-integral.
\end{lemma}

\begin{proof}
We'll prove the statement by contrapositive. Let $(w_1, ..., w_n)$ be in the fractional independence polytope and suppose that $w \not \in \{0,1/2,1\}^n$. Let $\epsilon_1 = \min_{w_i \not \in \{0,1/2,1\}} \min \{|w_i|, |w_i - 1/2|, |w_i - 1|\}$ and let

	\[\epsilon_2 = \min_{\substack{w_i, w_j\\v_i v_j \in E(G)\\ w_i + w_j \not = 1}} |1 - w_i - w_j|\]
	
	if there exists an edge $v_i v_j$ such that $w_i + w_j < 1$. If no such edge exists, we set $\epsilon_2 = 1$. Take $\epsilon = \min(\epsilon_1, \epsilon_2/2)$ and define $x,y \in \mathbb{R}^n$ by
	\[x_i = \begin{cases} w_i & w_i = 0, \frac{1}{2}, 1\\ w_i - \epsilon & w_i < 1/2 \\ w_i + \epsilon & w_i > 1/2, \end{cases}\]
	and 
	\[y_i = \begin{cases} w_i & w_i = 0, \frac{1}{2}, 1 \\ w_i + \epsilon & w_i < 1/2 \\ w_i - \epsilon & w_i > 1/2. \end{cases}\]
	We now claim that $x$ and $y$ are in the fractional independence polytope. First note that by our choice of $\epsilon_1$, all entries $x_i$ and $y_i$ are indeed in $[0,1]$. 
	
	We now show that $x$ is feasible. Let $v_i v_j$ be an edge in $G$. Note that if $w_i, w_j \leq 1/2$ then $x_i + x_j \leq w_i + w_j \leq 1$. So now suppose without loss of generality that $w_i > 1/2$. Then note that $w_j < 1/2$ and thus $x_i + x_j = w_i + w_j \leq 1$ as desired. 
	
	Similarly, to see that $y$ is feasible we first observe that if for some $v_i v_j \in E(G)$ we have $w_i + w_j < 1$, then $y_i + y_j \leq w_i + w_j + 2 \epsilon \leq w_i + w_j + \epsilon_2 \leq 1$. This immediately handles the case where $w_i \leq 1/2$ and $w_j < 1/2$. Since the case of $w_i = w_j = 1/2$ is trivial, we finally note that if $w_i > 1/2$, then $w_j < 1/2$ and $y_i + y_j = w_i + w_j \leq 1$. 
	
	So both $x$ and $y$ are feasible and clearly $w = (x+y)/2$. As $\epsilon > 0$, $x \not = w$ and $y \not = w$. We then conclude that $w$ is not a vertex.
	
\end{proof}

\section{Threshold Graph Maximization Results}

\subsection{A Local Move and Results for Dense Graphs}
Given a graph $G$ there is a natural way of slowly transforming it into a threshold graph (see Figure \ref{fig:localmove}): Take two vertices $u$ and $v$. Remove the edges between $u$ and $N(u) \setminus \overline{N(v)}$ and add edges from $N(u) \setminus \overline{N(v)}$ to $v$. This gives us a new graph $G'$ in which we have $N_{G'}(u) \subseteq \overline{N_{G'}(v)}$. Repeating this process takes us from $G$ to a threshold graph. Notationally, we refer to this operation on a graph $G$ as a local move on $(u,v)$. We will also simply refer to it as moving the neighbors of $u$ to $v$.

\begin{figure}
\centering
\begin{subfigure}{.5\textwidth}
  \centering
  \begin{tikzpicture}[scale=1, transform shape]

        \node[circle, inner sep=0pt, minimum size=2mm, fill, label={[align=flush center]below: \scalebox{1}{$u$}}] (u) at (0,-2) {};
    
        \node[circle, inner sep=0pt, minimum size=2mm, fill, label={[align=flush center]below: \scalebox{1}{$v$}}] (v) at (1.5,-2) {};
    
        \node (a) at (-1,0) {};
        \node (b) at (-.5,0) {};
        \node (c) at (.675,0) {};
        \node (d) at (1.5,0) {};
        \node (e) at (2,0) {};
    
        \draw (u) -- (a);
        \draw (u) -- (b);
        \draw (u) -- (c);
        \draw (v) -- (c);
        \draw (v) -- (d);
        \draw (v) -- (e);
        \draw (u) -- (v);
        
        \draw (0,0) circle [radius = 1.25cm] node[align=left, text width=2cm] {\scalebox{1}{$N(u)$}};

        \draw (1.25,0) circle [radius = 1.25cm] node[align=right, text width=2cm] {\scalebox{1}{$N(v)$}};
        \draw node at (.65,-.2) [above] {\scalebox{.5}{$N(u) \cap N(v)$}};
    
    \end{tikzpicture}
    \caption{Before applying the local move}
  
\end{subfigure}%
\begin{subfigure}{.5\textwidth}
  \centering
  \begin{tikzpicture}[scale=1, transform shape]
    
        \node[circle, inner sep=0pt, minimum size=2mm, fill, label={[align=flush center]below: \scalebox{1}{$u$}}] (u) at (0,-2) {};
    
        \node[circle, inner sep=0pt, minimum size=2mm, fill, label={[align=flush center]below: \scalebox{1}{$v$}}] (v) at (1.5,-2) {};
    
        \node (a) at (-1,0) {};
        \node (b) at (-.5,0) {};
        \node (c) at (.675,0) {};
        \node (d) at (1.5,0) {};
        \node (e) at (2,0) {};
    
        \draw (v) -- (a);
        \draw (v) -- (b);
        \draw (u) -- (c);
        \draw (v) -- (c);
        \draw (v) -- (d);
        \draw (v) -- (e);
        \draw (u) -- (v);
        
        \draw (0,0) circle [radius = 1.25cm] node[align=left, text width=2cm] {\scalebox{1}{$N(u)$}};

        \draw (1.25,0) circle [radius = 1.25cm] node[align=right, text width=2cm] {\scalebox{1}{$N(v)$}};
        \draw node at (.65,-.2) [above] {\scalebox{.5}{$N(u) \cap N(v)$}};
    
    \end{tikzpicture}
    \caption{After applying the local move}
\end{subfigure}
\caption{A neighborhood-ordering local move}
\label{fig:localmove}
\end{figure}
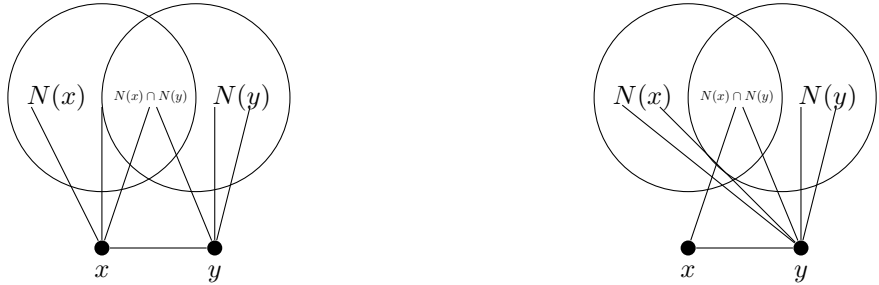

To prove Theorem \ref{thresholdMaximizationTheorem} we start by showing that this local move doesn't significantly decrease the number of homomorphisms from $H$. Towards this end, we define forbidden paths.

    \begin{definition}[Forbidden Path]
        A forbidden path is a path $wxyz$ where $wy$ and $xz$ are not edges.
    \end{definition}

We show that the homomorphisms lost after applying the local move are those that send a forbidden path in $H$ to $u$, $v$ and vertex in $N(u) \setminus \overline{N(v)}$.

    \begin{lemma}
    \label{localMoveLemma}
    Let $H$ and $G$ be graphs, $u$ and $v$ be vertices of $G$, and let $G'$ be the resulting graph after applying a local move on $(u,v)$. Then $\hom(H,G')$ is at least as large as the number of homomorphisms $\varphi \in \Hom(H,G)$ such that for all $w \in N(u) \setminus \overline{N(v)}$ and all forbidden paths $abcd$ in $H$, $\{u,v,w\} \not \subseteq \varphi(\{a,b,c,d\})$. 
        \end{lemma}

    \begin{proof} We proceed by creating an injective function that takes a homomorphism $\varphi(x)$ from $H$ to $G$ satisfying the assumptions in the theorem, and maps it to a homomorphism $\varphi'(x)$ from $H$ to $G'$. For intuition, we attempt to map $\varphi$ to a $\varphi'$ such that $\varphi'(x) = \varphi(x)$ whenever possible.
    
    	If $\varphi(x) \not = u,v$ then we simply let $\varphi'(x) = \varphi(x)$. If $\varphi(x) = u$, then we again let $\varphi'(x) = \varphi(x)$ so long as there is not an edge $xy \in E(H)$ with $\varphi(y) \in N(u) \setminus \overline{N(v)}$. If there is such an edge $xy \in E(H)$, then we instead set $\varphi'(x) = v$. Finally, if $\varphi(x) = v$ then we set $\varphi'(x)= \varphi(x)$ unless there is an edge $xy \in E(H)$ such that $\varphi(y) = u$ and $\varphi'(y) = v$. If such an edge exists, then we set $\varphi'(x) = u$. Formally,
    
    \[\varphi'(x) = \begin{cases}
                                    v & \exists y: xy \in E(H) \text{ and } \varphi(x) = u \text{ and } \varphi(y) \in N(u) \setminus \overline{N(v)} \\
                                    u & \exists y, z: xy, yz \in E(H) \text{ and } \varphi(x) = v \text{ and } \varphi(y) = u \text{ and } \varphi(z) \in N(u) \setminus \overline{N(v)} \\
                                    \varphi(x) & \text{otherwise}.
                        \end{cases}
        \]
    
     We start with the following claim.
     
     \begin{claim}
     $\varphi'$ is a homomorphism from $H$ to $G'$.	
     \end{claim}
	
	\begin{proof}
        
        It suffices to check that only the edges incident to vertices mapping to $u$ or $v$ are preserved. Let $a,b \in V(H)$ be such that $ab \in E(H)$ and $\varphi(a) = u$. Since $\varphi$ is a homomorphism, we must have that $\varphi(b) \in N(u)$. We now take cases.
    
    \begin{enumerate}[leftmargin=*, label=\emph{Case \Roman*.}]
    
    \item $\varphi(b) \in N(u) \cap N(v)$
        
         Since $\varphi(b) \not = u,v$, $\varphi'(b) = \varphi(b)$. Moreover, $\varphi'(a)\varphi'(b) \in E(G')$ since $a$ maps to either $u$ or $v$ under $\varphi'$ and both $u\varphi'(b)$ and $v\varphi'(b)$ are edges in $G'$. 

    \item $\varphi(b) \in N(u) \setminus \overline{N(v)}$.
        
        If we are in this case, then $\varphi'(a) = v$ as $ab \in E(H)$, $\varphi(a) = u$, and $\varphi(b) \in N(u) \setminus \overline{N(v)}$. Moreover, $\varphi'(b) = \varphi(b)$ as $\varphi(b) \not = u,v$. Since $v\varphi(b) \in E(G')$, we are again done.

    \item $\varphi(b) = v$ and $\varphi'(a) = u$  

        We claim that $\varphi'(b) = v$. Suppose not. Then $\exists y, z: by, yz \in E(H), \varphi(y) = u, \text{ and } \varphi(z) \in N(u) \setminus \overline{N(v)}$. Since $yz \in E(H)$, $\varphi'(y) = v$ and thus $y \not = a$. But now note $\varphi(a) = \varphi(y) = u$, so $ay \not \in E(H)$. Moreover, $bz \not \in E(H)$ since $\varphi(z) \in N(u) \setminus \overline{N(v)}$ and $\varphi(b) = v$. It then follows that $abyz$ is a path in $H$. But then $\varphi$ maps the forbidden path $abyz$ to $u,v,$ and a vertex in $N(u) \setminus \overline{N(v)}$, which is a contradiction.

    \item $\varphi(b) = v$ and $\varphi'(a) = v$

            Note that since $\varphi'(a) = v$ we must have that there is a $y$ such that $ay \in E(H)$ and $\varphi(y) \in N(u) \setminus \overline{N(v)}$. But it then follows that $\varphi'(b) = u$ since $ba, ay \in E(H)$. 

    \end{enumerate}
    
        Now let $a$ and $b$ be vertices such that $\varphi(a) = v$ and $ab \in E(H)$ where $\varphi(b) \not = u$. Then we have one of two cases
        
    \begin{enumerate}[leftmargin=*, label=\emph{Case \Roman*.}]
        
    \item $\varphi'(a) = v$
        
        Since $\varphi$ is a homomorphism we have that $\varphi(b) \in N_G(v)$. But then we have that the edge $ab$ is preserved by the homomorphism since $N_{G}(v) \subseteq N_{G'}(v)$. 

    \item $\varphi'(a) = u$
        
        In this case there exists $y, z$ such that $ay, yz \in E(H)$, $\varphi(y) = u$, and $\varphi(z) \in N(u) \setminus \overline{N(v)}$. Now we claim $\varphi(b) \in N(u)$. Suppose not, then $\varphi(b) \in N(v) \setminus \overline{N(u)}$. We now note that $bayz$ is a forbidden path: $by \not \in E(H)$ as $\varphi(y) = u$ and $\varphi(b) \not \in N(u)$ and $az \not \in E(H)$ as $\varphi(a) = v$ and $\varphi(z) \in N(u) \setminus \overline{N(v)}$. But then $bayz$ is forbidden path with $u,v$ and $\varphi(z) \in N(u) \setminus \overline{N(v)}$ in the image of $\varphi$, a contradiction.
        
    \end{enumerate}

\noindent Since these are all the cases we indeed have that $\varphi'$ is a valid homomorphism. 

\end{proof}
    
    To see that the map is injective, suppose that we had that $\varphi_1, \varphi_2 \mapsto \varphi'$. Clearly $\varphi_1(x) = \varphi_2(x)$ for all $x$ such that $\varphi'(x) \not = u,v$. Now suppose that for some $x \in V(H)$, $\varphi_1(x) = u$ and $\varphi_2(x) = v$. If $\varphi'(x) = v$, then $x$ is adjacent to some $y$ such that $\varphi_1(y) \in N(u) \setminus \overline{N(v)}$. But then $\varphi_2$ also maps $y$ to the same vertex, but $v$ is not adjacent to $\varphi_2(y)$ in $G$, which contradicts the fact that $\varphi_2$ is a homomorphism.
    
    So now suppose that $\varphi'(x) = u$. Then there exist $y, z$ such that $xy, yz \in E(H) \text{ and } \varphi_2(y) = u \text{ and } \varphi_2(z) \in N(u) \setminus \overline{N(v)}$.  But then we must have that $\varphi_1(y) = v$ since $xy \in E(H)$ and $\varphi_1(x) = u$. Similarly, we have that $\varphi'(y) = v$ as $xy \in E(H)$ and $\varphi'(x) = u$. But now $y$ is a vertex such that $\varphi_1(y) = v$ and $\varphi_2(y) = u$ and $\varphi'(y) = v$, which we just showed yields a contradiction.
    
    Hence the map is injective as desired and the statement holds.
    \end{proof}

    \begin{corollary}
        \label{exactMaxCorollary}
        Let $H$ and $G$ be graphs, $u,v \in V(G)$, and take $G'$ to be the resulting graph after applying a local move $(u,v)$ to $G$. If $H$ has no induced graphs isomorphic to $P_4$ or $C_4$, then $\hom(H,G') \geq \hom(H,G)$.
    \end{corollary}
        
    \begin{proof}
        If $H$ does not contain induced subgraphs isomorphic to $P_4$ and $C_4$, then we have that $H$ has no forbidden paths. Hence $H$ has at least as many homomorphisms to $G'$ as it does to $G$.
    \end{proof}
    
    In order to use Lemma \ref{localMoveLemma}, we'll need more quantitative bounds on the number of homomorphisms lost. Towards, this end we prove the following lemma.
    
    \begin{lemma}
        \label{badhomLemma}
        Let $H$ and $G$ be graphs with $n := |V(G)|$ and let $u,v \in V(G)$ and $S \subseteq V(G) \setminus \{u,v\}$. Then there are at most $O_H(|S| n^{|H|-3})$ homomorphisms $\varphi \in \Hom(H,G)$ such that there exists a forbidden path $abcd$ in $H$ with $u,v,s \in \varphi(\{a,b,c,d\})$ for some $s \in S$. 
    \end{lemma}
        
    \begin{proof}
	Let $abcd$ be a forbidden path in $H$. Then we must choose one of the four vertices to map to $u$, one of the remaining $3$ to map to $v$, and one of the remaining two to map to a vertex in $S$. The remaining $|H|-3$ vertices in the graph could go to any of $n$ vertices. Hence, we have that there are at most
            \[4! |S| n^{|H|-3}\]
        such homomorphisms. Since there are only a constant number of forbidden paths in $H$, which we denote by $k_H$, we have that there are at most
            \[4! k_H |S| n^{|H|-3} = O_H(|S| n^{|H|-3})\]
        homormophism mapping a forbidden path to set including $u$, $v$ and a vertex from $S$.
    \end{proof}
    
   	With this lemma in hand, we note that moving the neighbors of $u$ to $v$ removes at most $O_H(|N(u) \setminus \overline{N(v)}| n^{|H|-3})$. Thus to argue we didn't lose too many homomorphisms after applying multiple local moves we define
    
    \begin{definition}[Total Movement]
        Let $G = G_0, G_1, ..., G_t$ be a sequence of graphs from applying local moves $(u_1, v_1), ..., (u_t, v_t)$ i.e. each $G_i$ is obtained by applying the local move to the vertices $(u_i, v_i)$ in the graph $G_{i-1}$. Define the total movement after these $t$ moves as
            \[\sum_{i=1}^t \left |N_{G_{t-1}}(u_i) \setminus \overline{N_{G_{t-1}}(v_i)} \right|.\]
    \end{definition}
    
	We now claim that we can turn $G$ into a threshold graph with small amount of total movement.
        
    \begin{lemma}
    \label{transformationLemma}
    Let $G$ be a graph. Then using local moves we can turn $G$ into a threshold graph with at most $n^2$ moves and with total movement at most $|E(G)|$.
    \end{lemma}
    
    \begin{proof}
    We prove this by induction on the number of vertices. 
    
    For the base case, note that a graph with a single vertex is threshold. Now assume the statement holds for graphs with at most $n$ vertices and let $G$ be a graph on $n+1$ vertices. Let $v_{n+1}$ be the vertex of maximum degree and $v_1, ..., v_n$ be the remaining vertices. Consider applying the following $n$ moves: Move the neighbors of $v_1$ to $v_{n+1}$, then move the neighbors of $v_2$ to $v_{n+1}$, continue in this way until moving the neighbors from $v_{n}$ to $v_{n+1}$. 
    
    Let $G'$ be the graph after applying all the above moves. Clearly $N_{G'}(v_1) \cup N_{G'}(v_2) \cup ... \cup N_{G'}(v_n) \subseteq \overline{N_{G'}(v_{n+1})}$. Hence if $v_i v_j \in E(G)$ and $j \not = n+1$, then we have that $v_{n+1} v_j \in E(G')$. We then conclude that if $v_{n+1} v_j \not \in E(G')$ for $j \not = n+1$, then it must be the case that $v_j$ is isolated. 
    
    Let $I$ denote the vertices in $G'$ that are isolated. Applying the inductive hypothesis to $G'[V \setminus (I \cup \{v_{n+1}\})]$, we have that there is a set of at most $(n-1)^2$ moves that transforms $G'[V \setminus (I \cup \{v_{n+1}\})]$ to a threshold graph $T$ with total movement at most $|E(G'[V \setminus (I \cup \{v_{n+1}\})])|$. Applying these moves to $G'$ gives us a graph $G''$. But now note that $G''$ can be described as adding a dominating vertex to $T$ followed by $|I|$ isolated vertices. So $G''$ is a threshold graph. Note that we made at most $(n-1)^2 + n \leq n^2$ moves. Moreover note that any edge incident to $v_{n+1}$ was moved at most once by the initial set of $n$ moves. Hence we incur at most $|E|$ movement cost, as desired.
    \end{proof}
    
    We can now prove a strengthening of Theorem \ref{thresholdMaximizationTheorem}. 
    
    \begin{theorem}
    \label{quantitativeThresholdMaximizationTheorem}
    	For any graph $G$ with $n$ vertices and any graph $H$, there exists a threshold graph $T$ on $n$ vertices with $|E(T)| \leq |E(G)|$ such that
            \[\hom(H,T) \geq \hom(H,G) - O_{H}(n^{|H|-1}). \]
    \end{theorem}

    \begin{proof}
		 We use the local moves to transform the true maximum graph into a threshold graph and argue that we haven't lost too many homomorphisms. By Lemma \ref{transformationLemma}, there exists a sequence of local moves on pairs $(u_1, v_1)$, ..., $(u_k, v_k)$ with total movement at most $|E(G)|$ that transforms $G$ into a threshold graph $T$. Let $G_i$ be the graph defined by applying $(u_i, v_i)$ to $G_{i-1}$. By Lemmas \ref{localMoveLemma} and Lemma \ref{badhomLemma}, we have that	
		 	\[\hom(H, G_i) \geq \hom(H, G_{i-1}) - O_H(|N(u_i) \setminus \overline{N(v_i)}| n^{|H|-3} ).\]
		 We thus conclude that
		 	\[\hom(H,T)\geq \hom(H,G) - \sum_{i=1}^k O_H(|N(u_i) \setminus \overline{N(v_i)}| n^{|H|-3} ) \geq \hom(H,G) - O_H( n^{|H|-1} ) \]
		 where the final inequality uses that the total movement of our sequence of local moves is at most $|E(G)| \leq n^2$. 
    \end{proof}
    
    \begin{remark}
        Note that such a result will not hold non-asymptotically for $\mathcal{M}(H,n,m)$. Specifically, we note that no threshold graph achieves $\mathcal{M}(C_4, 4, 4)$. 
    \end{remark}
    
\subsection{Extension To Sparse Graphs}
We now extend this result to sparse graphs by strengthening Lemma \ref{badhomLemma}. To do this, we rely on a result of \cite{janson2004upper}.

\begin{theorem}[Janson et al. \cite{janson2004upper}]
    \label{subgraphCountTheorem}
    Let $H$ be a graph and $m \geq |E(H)|$ and $n \geq |H|$ with $n \leq m \leq \binom{n}{2}$. Then 
        \[\mathcal{M}(H,n,m) = \Theta_H(m^{|H| - \alpha^*(H)} n^{2 \alpha^*(H) - |H|}), \]
    where $\alpha^*(H)$ denotes the fractional independence number of the graph.
\end{theorem}

From this, we easily obtain the following lemma:

\begin{lemma}
\label{subgraphHoms}
Let $H$ be a graph, $k$ be a non-negative integer, and $H'$ be an induced subgraph of $H$ on $|H|-k$ vertices. Then for every $n \geq |H|$ and $n \leq m \leq \binom{n}{2}$ with $m \geq |E(H)|$
	\[\mathcal{M}(H',n,m) = O_H \left(\mathcal{M}(H,n,m) /\left(\frac{m}{n} \right)^k \right). \]
\end{lemma}

\begin{proof}
Let $\Delta = \alpha^*(H) - \alpha^*(H') \geq 0$. By Theorem \ref{subgraphCountTheorem} and some algebraic manipulation we see that
\begin{align*}
\mathcal{M}(H',n,m) & = O_H \left( \left(\frac{m}{n} \right)^{|H'| - \alpha^*(H')} n^{\alpha^*(H')} \right) \\
& = O_H \left( \left(\frac{m}{n} \right)^{|H| - k - \alpha^*(H) + \Delta} n^{\alpha^*(H) - \Delta} \right) \\
& = O_H \left( \left(\frac{m}{n}\right)^{-k} \left( \frac{m}{n^2} \right)^\Delta \left( \frac{m}{n} \right)^{|H| - \alpha^*(H)} n^{\alpha^*(H)} \right) \\
& = O_H(\mathcal{M}(H,n,m)/\left(\frac{m}{n} \right)^k).
\end{align*}

where the last equality used that $\mathcal{M}(H,n,m) = \Omega_H(m^{|H|-\alpha^*(H)} n^{2 \alpha^*(H) - |H|})$ and $m \leq n^2$.

\end{proof}

Note that the above lemma can be proved without the result of Theorem \ref{subgraphCountTheorem} by taking a graph that approximately achieves $\mathcal{M}(H',n/2,m/4)$ and adding $\Omega(m/n)$ dominating vertices to it to get a graph with many copies of $H$. However, the above proof is simpler to present.

\begin{lemma}
\label{sparsebadHomLemma}
Let $H$ and $G$ be graphs, where $G$ has $n \geq |H|$ vertices and $m \geq \max(|E(H)|, n)$ edges, and let $u,v \in V(G)$ and $S \subseteq V(G) \setminus \{u,v\}$. Then we have that there are at most $O_H(|S| \mathcal{M}(H,n,m)/\left(\frac{m}{n} \right)^3)$ homomorphisms $\varphi \in \Hom(H,G)$ such there exists a forbidden path $abcd \in H$ in which $\{u,v,s\} \subseteq \varphi(\{a,b,c,d\})$ for some $s \in S$.
\end{lemma}

\begin{proof}
Let $abcd$ be a forbidden path. Then we have that there are clearly $O(|S|)$ ways to map $3$ vertices from $a,b,c,d$ into $u,v$ and some vertex from $|S|$. For ease of notation, assume these three vertices are $a,b,c$. Now consider $H' = H[V(H) \setminus \{a,b,c\}]$. We note that the number of homomorphisms is at most $O_H(|S| \mathcal{M}(H',n,m))$. So it suffices to show that $\mathcal{M}(H',n,m) = O_H \left(\mathcal{M}(H,n,m)/\left(\frac{m}{n} \right)^3 \right)$. But this follows from Lemma \ref{subgraphHoms}.
\end{proof}

We can now prove our strengthening of Theorem \ref{quantitativeThresholdMaximizationTheorem}.

\begin{theorem}
    For any graph $H$ and integers $n$ and $m \leq \binom{n}{2}$, there exists a threshold graph $T$ on $n$ vertices and at most $m$ edges such that
            \[\hom(H,T) \geq \left( 1 - O_{H} \left (\frac{n^3}{m^2} \right ) \right) \mathcal{M}(H,n,m). \]
\end{theorem}

\begin{proof}
The proof follows almost identically to Theorem \ref{thresholdMaximizationTheorem}. Let $G^\star$ be the graph on $n$ vertices and $m$ edges with
	\[\hom(H,G^\star) = \mathcal{M}(H, n, m).\]

 By Lemma \ref{transformationLemma}, there exists a sequence of local moves on pairs $(u_1, v_1)$, ..., $(u_k, v_k)$ with total movement at most $|E(G^\star)|$ that transforms $G^\star$ into a threshold graph $T$. Let $G_i^\star$ be the graph defined by applying $(u_i, v_i)$ to $G_{i-1}^\star$. By Lemmas \ref{localMoveLemma} and Lemma \ref{sparsebadHomLemma}, we have that	
		 	\[\hom(H, G_i^\star) \geq \hom(H, G_{i-1}^\star) - O_H \left(|N(u_i) \setminus \overline{N(v_i)} | \cdot \mathcal{M}(H,n,m)/ \left(\frac{m}{n} \right)^3 \right).\]
		 We thus conclude that
		 	\begin{align*}
		 		\hom(H,T) &\geq \hom(H,G^\star) - \sum_{i=1}^k O_H(|N(u_i) \setminus \overline{N(v_i)} | \cdot \mathcal{M}(H,n,m)/ \left(\frac{m}{n} \right)^3 ) \\
		 		 &\geq \mathcal{M}(H,n,m) - O_H \left( \frac{n^3}{m^2} \cdot \mathcal{M}(H,n,m) \right)
		 	\end{align*}
		 where the final inequality uses the total movement of our moves is $|E(G)|$. 
\end{proof}

\subsection{Extension to Hypergraphs}
For hypergraphs, we only prove results for asymptotic maximization in the dense case. As before, we use a sequence of moves to transform our hypergraph. We can view this as a shifting or compression argument, which are common in extremal set theory.

\begin{definition}[Hypergraph Move]
Let $G = (V,E)$ be a $k$-uniform hypergraph and $u,v$ be two vertices in $G$. We say a \textit{hypergraph move} from $u$ to $v$ in $G$ is as follows: For each hyperedge $e \in E(G)$ containing $u$ and not $v$ such that $(e \setminus \{u\}) \cup \{v\} \not \in E(G)$, remove $e$ and add $(e \setminus \{u\}) \cup \{v\}$.	
\end{definition}

Note that this move generalizes our local move for graphs. That said, we will be able to get by with a worse bound on the number of homomorphisms lost from such a hypergraph move: 

\begin{lemma}
    \label{localMoveHypergraphLemma}
    Let $H$ and $G$ be $k$-uniform hypergraphs and let $G'$ be the resulting hypergraph after applying a hypergraph move from $u$ to $v$ in $G$. Then the number of homomorphisms from $H$ to $G'$ is at least the number of homomorphisms $\varphi$ from $H$ to $G$ that do not use both $u$ and $v$.
\end{lemma}

\begin{proof}
We again construct an injective map from a subset of $\hom(H,G)$ to $\hom(H,G')$. Specifically, given $\varphi \in \hom(H,G)$, we map it to

\[\varphi'(x) = \begin{cases} 
	v & \varphi(x) = u \text{ and }\exists e \in E(H): x \in e \text{ and } (\varphi(e) \setminus \{u\}) \cup \{v\} \not \in E(G) \\
	\varphi(x) & \text{otherwise}.
\end{cases}\]

We can easily verify that $\varphi'$ is a homomorphism if $\varphi$ doesn't use both $u$ and $v$. Moreover, suppose $\varphi_1, \varphi_2 \in \hom(H,G)$ both map to a homomorphism $\varphi' \in \hom(H,G')$. Towards a contradiction, suppose $\varphi_1 \not = \varphi_2$. By definition, $\varphi_1(x) = \varphi_2(x) = \varphi'(x)$ for all $x$ such that $\varphi_1(x) \not = u$ and $\varphi_2(x) \not = u$. So without loss of generality there exists an $x$ such that $\varphi_1(x) = u$ and $\varphi_2(x) = v$. Then $\varphi'(x) = v$ and there exists an $e \in E(H)$ containing $x$ such that $(\varphi_1(e) \setminus \{u\}) \cup \{v\} \not \in E(G)$. But now note that $\varphi_2(e) = (\varphi_1(e) \setminus \{u\}) \cup \{v\} \not \in E(G)$, so we have reached a contradiction. 
\end{proof}

The key lemma of this section will be the following result to transform hypergraphs into threshold hypergraphs:

\begin{lemma}
\label{lem:hypergraphTransformation}
Let $G$ be a $k$-uniform hypergraphs with $n$ vertices, then $G$ can be transformed into a threshold hypergraph $T$ by removing $o_k(n^k)$ edges and using $o_k(n^2)$ hypergraph moves.
\end{lemma}

This quickly gives Theorem \ref{hypergraphThresholdMaximizationTheorem}.

\hypergraphThresholdMaximizationTheorem*

\begin{proof}
Let $G$ be a hypergraph with $n$ vertices. By Lemma \ref{lem:hypergraphTransformation} we can use $o_k(n^2)$ hypergraph moves $(u_1, v_1), ..., (u_\ell, v_\ell)$ and remove at most $o_k(n^k)$ edges to turn it into a threshold hypergraph $T$. Now note the number of homormorphisms $H$ to any hypergraph $G'$ that use any two specific vertices $u,v \in V(G')$ is at most  $O_H(n^{|H|-2})$. Similarly, there are at most $O_{H,k}(n^{|H|-k})$ homomorphisms that use every vertex in a hyperedge $e \in E(G')$. Thus by Lemma \ref{localMoveHypergraphLemma} each hypergraph move $(u_i, v_i)$ removes at most $O_H(n^{|H|-2})$ homomorphisms. Moreover, removing a hyperedge $e$ removes at most $O_{H,k}(n^{|H|-k})$ homomorphisms. We conclude that 
		\[\hom(H,T) \geq \hom(H,G) - o_k(n^2) \cdot O_H(n^{|H|-2}) - o_k(n^k) \cdot O_{H,k}(n^{|H|-k}) = \hom(H,G) - o_{H,k}(n^{|H|}). \] 
\end{proof}

We now proceed to prove Lemma \ref{lem:hypergraphTransformation}. Like in the case of graphs, we will apply a set of local moves to $G$, yielding a graph $G'$, in such a way that there exists a vertex $v \in G'$ with $u \ll_{G'} v$ for all $u \in V(G')$. We then repeat the process on $G'$ to slowly transform it into a threshold hypergraph. One way to do this would of course be to apply a local move from every vertex $u \in V(G)$ to $v$, however, this would result in a total of $\Omega(n^2)$ moves. To improve on this, we'll show that we only need to apply moves from a neighborhood-dominating set to $v$.

\begin{definition}[Neighborhood-Dominating Set]
	Given a $k$-uniform hypergraph $G=(V,E)$, a vertex $v \in V$, and a subset $S \subseteq V$, a neighborhood-dominating set $D \subseteq V$ of $S$ with respect to $v$ is a set of vertices such that for every vertex $s \in S$ and hyperedge $e$ containing $s$ either $v \in e$ or there exists a $d \in D$ such that $(e \setminus \{s\}) \cup \{d\} \in E(G)$.
\end{definition}

Of course, this is only useful if there are neighborhood-dominating sets of size $o(n)$. To prove that small neighborhood-dominating sets exist, we'll apply the following lemma to a suitably defined graph corresponding to our hypergraph $G$.

\begin{lemma}
\label{lem:DomSet}
Let $G = (A \cup B, E)$ be a bipartite graph with $|A| = n$ and $|B| \leq n^k$ for a positive integer $k$. Moreover, assume that the minimum degree of a vertex in $B$ is $\delta$, then there exists a set $D \subseteq A$ such that $N_G(D) = B$ and $|D| = O_k (n\log(n)/\delta)$.
\end{lemma}

\begin{proof}
Assume that $\delta > k \log(n)$ as otherwise we can simply take $D = A$. We will follow the classical proof for dominating sets given in \cite{alon2004probabilistic}. Take a random set $S$ such that every vertex of $A$ is in $S$ independently with probability $p$ and let $N$ denote the set of vertices in $B$ without a neighbor in $S$. Build a set $D$ by adding a neighbor of each vertex in $N$ to $S$. Then observe that $N_G(D) = B$ and in expectation we have 
	\[\E[|D|] \leq \E[|S| + |N|] \leq pn + n^k (1-p)^{\delta} \leq pn + n^k e^{-p \delta}.\]

Taking $p = \frac{1}{\delta} \log(\delta n^{k-1})$ then gives
	\[E[|D|] \leq \frac{kn\log(n)}{\delta} + \frac{n}{\delta} = O_k \left( \frac{n\log(n)}{\delta} \right).\]
\end{proof}

We now associate some graphs to our hypergraphs:

\begin{definition}[Incidence Graph]
	Given a $k$-uniform hypergraph $G = (V,E)$, we associate a bipartite incidence graph $I = (A \cup B, E(I))$. Each vertex in $A$ will correspond to a vertex in $G$ and each vertex in $B$ corresponds to a subset $S \in \binom{[n]}{k-1}$ such that $S \subset e \in E(G)$. There is an edge between a vertex $v \in A$ and a subset $S \in B$ if $\{v\} \cup S \in E(G)$.
\end{definition}

In particular, we will often be interested in induced subgraphs of the incidence graph with respect to a vertex $v$.

\begin{definition}[Induced Incidence Graphs]
Given a $k$-uniform hypergraph $G = (V,E)$, a set of vertices $S \subseteq V$, and a $v \in V$, we say the the incidence graph induced by $S$ with respect to $v$ is $I[S \cup N_I(S) \setminus \{b \in \{\binom{[n]}{k-1}: v \in b \text{ or } b \cup v \in E\}]$, where $I = (A \cup B, E(I))$ is the incidence graph of $G$.
\end{definition}

With these definitions in hand, we get the following corollary. 

\begin{corollary}
\label{cor:neighborhoodDomSet}
	Let $G = (V,E)$ be a $k$-uniform hypergraph on $n$ vertices, $S \subseteq V$, $v \in V$, and $I= (A \cup B, E(I))$ be the incidence graph of $G$ induced by $S$ with respect to $v$. If every vertex $b \in B$ has degree at least $n^{2/3}$, then there exists a neighborhood-dominating set $D \subseteq S$ of $S$ with respect to $v$ of size at most $O_k(\sqrt{n})$.
\end{corollary}

\begin{proof}
	First note that if $|S| < \sqrt{n} + 1$ we can take $D = S$. So we assume $|S| \geq \sqrt{n} + 1$. We then have that $|B| \leq n^{k-1} \leq |A|^{2k-2}$. By Lemma \ref{lem:DomSet}, it follows that there exists a set $D \subseteq S$ of size at most $O_k(n \log(n) / n^{2/3}) = O_k(\sqrt{n})$ such that $N_I(D) = B$.
	
	We claim that $D \cup \{v\}$ is a dominating set of $S$ with respect to $v$. Indeed, let $s \in S$ and $e$ be an edge containing $s$ with $v \not \in e$. If $e \setminus \{s\} \cup \{v\} \in E$ then we are done, so suppose this is not the case. It then follows that $e \setminus \{s\} \in B$, so there exists a vertex $d \in D$ that is a neighbor of $e \setminus \{s\}$ in $I$, which implies $(e \setminus \{s\}) \cup \{d\} \in E(G)$, as desired.
\end{proof}

Our general scheme to transform our hypergraph will now be to remove edges of low degree so the induced incidence graph has high minimum degree and then apply local moves from our dominating sets to $v$. We will then repeat this process to find a dominating vertex $u$ of $V \setminus \{v\}$.

Formally, we define the following domination procedure from $S$ to $v$: Let $G_0$ be a hypergraph, $S \subset V(G_0)$, and $v \in V(G_0)$. Define a sequence of graphs $G_0, G_1, ..., G_k$ as follows: Let $I_i = (A_i \cup B_i, E(I_i))$ be the incidence graph of $G_i$ induced by $S$ and with respect to $v$. While there exists a $b \in B_i$ with $\deg(b) \leq n^{2/3}$ remove all edges of the form $b \cup \{s\}$ for some $s \in S$ from $G_i$ to get a graph $H_i$. Let $D_i$ be the smallest dominating set of $S$ with respect to $v$ in $H_i$. Apply local moves from each vertex $d \in D_i$ to $v$, yielding $G_{i+1}$.

\begin{lemma}
\label{lem:dominating-set-moves}
After running the domination procedure from $S$ to $v$ we have that $s \ll_{G_k} v$ for all $s \in S$. 
\end{lemma}

\begin{proof}
Fix $s \in S$ and an edge $e \in E(G_k)$ containing $s$ such that $v \not \in e$. We'll show that $e \setminus \{s\} \cup \{v\} \in E(G_k)$. First, observe that we never remove hyperedges containing $v$ as no subsets $b$ in the induced incidence graph contain $v$ or are neighbors of $v$. Since we only add edges incident to $v$, $e \in E(G_i)$ for all $i$. It then follows that there exists $d_1, ..., d_k$ such that $e \setminus \{s\} \cup \{d_i\} \in E(H_i)$. Note that if $e \setminus \{s\} \cup \{d_i\}$ exists when we apply the local move from $d_i$ to $v$, then clearly $e \setminus \{s\} \cup \{v\} \in G_{i+1}$. As we never remove edges containing $v$, $e \setminus \{s\} \cup \{v\} \in G_{k}$.

	Now towards a contradiction, suppose that for each $i$ we applied a local move from $d_i' \in D_i$ to $v$ before the move from $d_i$ to $v$, and the move from $d_i'$ causes the removal of $e \setminus \{s\} \cup \{d_i\}$. Observe that $d_i' \not = d_j'$ for $i < j$. Indeed suppose $d_i' = d_j'$. Note $d_i' \ll_{G_{i+1}} v$. Since we only add edges incident to $v$ and never remove edges containing $v$, we have that $d_i' \ll_{H_{\ell}} v$ for all $\ell > i$. In particular, at the time of making the move from $d_j'$ to $v$ we have that $d_j' \ll v$, so the local move from $d_j'$ to $v$ would remove no edges, a contradiction. 
		
	But we now observe that $\{d_1', ..., d_k'\} \subseteq e \setminus \{s\}$, a contradiction as $|e \setminus \{s\}| = k - 1$.
\end{proof}

It now remains to show that removing edges of low degree and applying local moves does not ruin the order we have built.

\begin{lemma}
\label{lem:subgraphMoves}
Let $G$ be a $k$-uniform hypergraph and let $u,v,w,x \in V(G)$. Moreover, assume $x \not = u,v$, $u \ll_G x$, $v \ll_G x$, and $w \ll_G x$. Let $G'$ be the resulting hypergraph after applying a hypergraph move from $u$ to $v$ in $G$. Then $w \ll_{G'} x$.
\end{lemma}

\begin{proof}

We start with the case that $w = v$. 

\begin{claim}
	$v \ll_{G'} x$
\end{claim}

\begin{proof}

Let $e \in E(G')$ be a hyperedge containing $v$ and not containing $x$. We'll show that $f = (e \setminus \{v\}) \cup \{x\} \in E(G')$. We now take two cases.

\begin{enumerate}[leftmargin=*, label=\emph{Case \Roman*.}]
    \item $e \in E(G)$
    
    Since $v \ll_G x$, we must have that $f \in E(G)$. If $u \not \in f$, then it cannot be removed by the local move and $f \in E(G')$. On the other hand, if $u \in f$ then $u \in e$. But now since $u \ll _G x$, $g = (e \setminus \{u\}) \cup \{x\} \in E(G)$. But $g = (f \setminus \{u\}) \cup \{v\}$, which implies $f$ isn't removed in the local move and $f \in E(G')$.
    \item $e \not \in E(G)$

    	Since $e \not \in E(G)$, it must have been added during the hypergraph move. Thus, $(e \setminus \{v\}) \cup \{u\} \in E(G)$. Since $u \ll_G x$, $(e \setminus \{v\}) \cup \{x\} = f \in E(G)$. Now note that $u \not \in f$ as $u \not \in e$. But then $f$ cannot be removed during the hypergraph move and $f \in E(G')$.
\end{enumerate}

\end{proof}

Since after applying the hypergraph move we have that $u \ll_{G'} v$, we have also shown that $u \ll_{G'} x$. So we now handle the case of $w \not = u,v$. 

\begin{claim}
	$w \ll_{G'} x$ for $w \not = u,v$
\end{claim}

\begin{proof}
	Let $e \in E(G')$ be an edge containing $w$ and not containing $x$. We'll again show $f = (e \setminus \{w\}) \cup \{x\} \in E(G')$ via cases on whether or not $e$ is in $E(G)$. 

\begin{enumerate}[leftmargin=*, label=\emph{Case \Roman*.}]
	\item $e \in E(G)$

		As $w \ll_G x$, $f \in E(G)$. Now, since $e$ wasn't removed it either doesn't contain $u$, contains both $u$ and $v$, or $(e \setminus \{u\}) \cup \{v\} \in E(G)$. If $u \not \in e$, then $u \not \in f$ and $f \in E(G')$, since we only remove edges containing $u$ in the hypergraph move. Similarly, if $e$ contains $u$ and $v$, then so does $f$ and we again have $f \in E(G')$ as we don't remove edges with both $u$ and $v$ in the hypergraph move. So $(e \setminus \{u\}) \cup \{v\} \in E(G)$. As $w \ll_G x$, we see $(e \setminus \{u, w\}) \cup \{v, x\} = (f \setminus \{u\}) \cup \{v\} \in E(G)$. But this again implies and $f \in E(G')$.
		
	\item $e \not \in E(G)$

	In this case, $e$ must be added during the hypergraph move and thus $v \in e$. Moreover since $e$ was added, $(e \setminus \{v\}) \cup \{u\} \in E(G)$. Since $w \ll_G x$, $g = (e \setminus \{v, w\}) \cup \{u, x\} \in E(G)$. Now note that after the hypergraph move we must have that $(g \setminus \{u\}) \cup \{v\} = f \in E(G')$.
	
\end{enumerate}
	
\end{proof}

\end{proof}

\begin{lemma}
\label{lem:lowDegreeRemoval}
	Let $G = (V,E)$ be a $k$-uniform hypergraph, $S \subseteq V$, $v \in S$, and $I = (A \cup B, E(I))$ be the incidence graph induced by $S$ with respect to $v$. Let $x \in V(G)$ and $y \in V(G) \setminus S$ be such that $x \ll_G y$ and $s \ll_G y$ for all $s \in S$. Suppose that for every subset $b \in B$ with $\deg_I(b) \leq n^{2/3}$, we remove the edges $b \cup \{s\}$ for $s \in S$ from $E(G)$, resulting in a hypergraph $G'$. Then $x \ll_{G'} y$.
\end{lemma}

\begin{proof}
	Let $e \in E(G')$ be an edge containing $x$ and not $y$. We'll show that $f = e \setminus \{x\} \cup \{y\} \in E(G')$. Since $E(G') \subseteq E(G)$, $e \in E(G)$. Combining this with the fact that $x \ll_G y$ implies $f \in E(G)$. Now fix some $s \in f \cap S \subseteq e \cap S$. We'll show that $f \setminus \{s\}$ does not lead to the removal of $f$. To do so, we take cases on why $e \setminus \{s\}$ didn't lead to the removal of $e$.
	
	\begin{enumerate}[leftmargin=*, label=\emph{Case \Roman*.}]
		\item $\deg_I(e \setminus \{s\}) > n^{2/3}$

			Note that $s \not = x$ as $x \not \in f$. Combining this with the fact that $x \ll_G y$, we observe that for any $t \not = y$, if $e \setminus \{s\} \cup \{t\} \in E(G)$ then $e \setminus \{s, x\} \cup \{t,y\} = (f \setminus \{s\}) \cup \{t\} \in E(G)$. So, $\deg_I(f \setminus \{s\}) \geq \deg_I(e \setminus \{s\}) > n^{2/3}$, and $f \setminus \{s\}$ doesn't lead to the removal of $f$.
			
		\item $v \in e$
	
		If $x \not = v$, then $v \in f$. Observe that by construction, the incidence graph has no edges corresponding to hyperedges in $G$ containing $v$. So we never remove edges containing $v$, and $f \in E(G')$. 
		
		Now assume that $x = v$. Since $s \ll_G y$, we have that $e \setminus \{s\} \cup \{y\} \in E(G)$. But now note $f \setminus \{s\} \cup \{v\} = e \setminus \{s\} \cup \{y\}$. It then follows that $f \setminus \{s\} \not \in B$. Thus, we conclude $f \setminus \{s\}$ doesn't lead to the removal of $f$.
		
		\item $e \setminus \{s\} \cup \{v\} \in E(G)$

			Since $x \ll_G y$, $e \setminus \{s,x\} \cup \{v,y\} \in E(G)$. But $e \setminus \{s,x\} \cup \{v,y\} = f \setminus \{s\} \cup \{v\}$. Thus $f \setminus \{s\} \not \in B$ and again doesn't cause the removal of $f$.
	\end{enumerate}
	
	Since we chose $s$ arbitrarily, $f \in E(G')$, as desired.

\end{proof}

With this, we can now prove Lemma \ref{lem:hypergraphTransformation}.

\begin{proof}[Proof of Lemma \ref{lem:hypergraphTransformation}]

	We apply the following algorithm to transform $G$ into a threshold hypergraph: 
	
	Initially, let $T = \emptyset$. Repeat $n$ times: Choose a vertex $v \in V \setminus T$. Apply the domination procedure with $G_0 = G$ from $V \setminus T$ to $v$, resulting in a graphs $G_1, G_2, ..., G_k$, intermediate graphs $H_0, ..., H_{k-1}$, and dominating sets $D_0, ..., D_{k-1}$. Set $T = T \cup \{v\}$ and $G = G_k$.
	
	\begin{claim}
		When the algorithm terminates, $G$ is threshold hypergraph.
	\end{claim}
	
	\begin{proof}
		Suppose that at time $r$ we have that $T = \{t_1, ..., t_r\}$, where $t_i$ is the $i$th vertex added to $T$. We'll prove by induction by induction on $r$ that $t_1 \gg_{G} t_2 \gg_G ... \gg_G t_r \gg s$ for all $s \in V(G) \setminus T$. The base case of $r = 0$ is vacuosly true. So now suppose the statement is true for some $r$. By Lemma \ref{lem:lowDegreeRemoval}, we have that $t_1 \gg_{H_0} t_2 \gg_{H_0} ... \gg_{H_0} t_r \gg_{H_0} s$ for all $s \in V(G) \setminus T$. By Lemma \ref{lem:subgraphMoves}, after applying the hypergraph moves we have $t_1 \gg_{G_1} t_2 \gg_{G_1} ... \gg_{G_1} t_r \gg_{G_1} s$ for all $s \in V(G) \setminus T$. Repeating this argument $k$ times gives us that $t_1 \gg_{G_k} t_2 \gg_{G_k} ... \gg_{G_k} t_r \gg_{G_k} s$ for all $s \in V(G) \setminus T$. Finally, by Lemma \ref{lem:dominating-set-moves} we have that $v \gg_{G_k} s$ for all $s \in V(G) \setminus (T \cup v)$. As $t_{r+1} = v$, we see that the inductive hypothesis holds.
	\end{proof}

	\begin{claim}
		Each iteration applies at most $O_k(\sqrt{n})$ local moves.
	\end{claim}
	
	\begin{proof}
		Note that by Corollary \ref{cor:neighborhoodDomSet}, every dominating set $D_i$ has size at most $O_k(\sqrt{n})$. It then follows that each iteration uses at most $O_k(k\sqrt{n}) = O_k(\sqrt{n})$ local moves.
	\end{proof}

	Note that the above claim implies that we make $O_k(n^{3/2})$ hypergraph moves in total. 

	\begin{claim}
		The algorithm removes $o_k(n^k)$ hyperedges in total.
	\end{claim}

	\begin{proof}
		Fix some $b \in \binom{[n]}{k-1}$. We claim that $b$ only causes the removal of edges at most $k$ times. Indeed, after removing all edges of the form $b \cup \{s\}$, we must add some hyperedge containing $b$ for $b$ to cause the removal of more hyperedges. Let $e$ be the first hyperedge containing $b$ added after $b$'s $\ell$th removal. $e$ must have been added by some local move, say from $s$ to $v$. Now since $e$ is the first such edge we have that $b \not \subset e \setminus \{v\} \cup \{s\}$. So $v \in b$. Since we make local moves to any vertex at most once and there are $k-1$ elements in $b$, it follows that $b$'s low degree only causes the remove of hyperedges at most $k$ times.(Note that we are using the fact that if $b$ is added due to a local move to $v$, it cannot be removed until the next iteration as we never remove hyperedges containing $v$ in that iteration.)
		
		We now conclude that each $b \in \binom{[n]}{k-1}$ causes the removal of at most $kn^{2/3}$ edges. So in total we remove at most 
			\[\binom{n}{k-1} k n^{2/3} = o_k(n^k)\]
		hyperedges in total.
	\end{proof}

\end{proof}

We end this section by remarking that the ideas in this proof simplify considerably in the case of graphs giving an alternative proof of Theorem \ref{thresholdMaximizationTheorem}. That said this argument loses slightly more homomorphisms than Theorem \ref{quantitativeThresholdMaximizationTheorem} and yields a weaker version of Theorem \ref{sparseThresholdMaximizationTheorem}.

\section{Applications of Threshold Maximization}
\subsection{Maximization on the Clique}

Before we begin with the applications, we start by reproving a result of \cite{gerbner2018maximum} that shows for any graph $H$ the quasi-clique maximizes the number of homomorphisms from $H$ for $c$ sufficiently large. This will contrast with the results of Section \ref{subsection:threeparts}, where we show that for small edge density $c$ there are graphs $H$ whose optimizers are threshold graphs that must have strictly more than two parts.

\begin{theorem}
\label{cliqueMaximizationTheorem}
Let $H$ be a fixed graph, then we have that for $c > k_H$, where $k_H \in [0,1)$, $\mathcal{M}_H(c)$ is achieved on the quasi-clique.
\end{theorem}

\begin{proof}
We will prove the result for a connected graph $H$. If the graph is not connected, we can apply the result to each of the components and take the maximum of the $k_H$'s over all the components. 

Let $T$ be a spanning tree of $H$. We note that clearly for any graph $G$, $t(H,G) \leq t(T,G)$ since $T$ was obtained by removing edges from $H$. Now using a result from Sidorenko \cite{sidorenko1994partially} we have that $t(T,G) \leq t(K_{1, |H|-1}, G)$ for all $G$. By a result of Reiher and Wagner \cite{reiher2018maximum}, we have that there exists a $k_H$ such that for $c > k_H$, $\mathcal{M}_{K_{1, |H|-1}}(c)$ is achieved on the quasi-clique. But now note that for any quasi-clique $\mathcal{K}$ have that $t(H, \mathcal{K}) = t(K_{1, |H|-1}, \mathcal{K})$, so $\mathcal{M}_H(c)$ is also attained on the quasi-clique.
\end{proof}

\subsection{Graphs Requiring Three Parts}
\label{subsection:threeparts}

In this section, we show that a result of Janson et al. \cite{janson2004upper} (cf. Theorem \ref{subgraphCountTheorem}) implies that some graphs require at least $3$ parts to optimize, which disproves a conjecture of Nagy that all graphs require only two parts \cite{nagy2017number}.

\begin{theorem}
\label{threePartsTheorem}
Let $H$ be a graph with $\alpha^*(H) > \alpha(H)$ and $\alpha^*(H) > \frac{|H|}{2}$, then for $c$ sufficiently small, $H$ is not optimized on the quasi-star or quasi-clique.
\end{theorem}

To begin, we show that for all graphs $H$ there exists a threshold graph $T$ on three parts such that $t(H,T)$ matches the upper bound from Theorem \ref{subgraphCountTheorem} up to constant factors depending on $H$. We will then prove Theorem \ref{threePartsHoms} by showing threshold graphs with at most two parts have significantly fewer homomorphisms.

\begin{lemma}
\label{threePartsHoms}
Let $H$ be a fixed graph and $n$ and $m$ be positive integers such that $2n \leq m \leq \binom{n}{2}$, then there exists a threshold graph $T$ with at most $3$ parts on $n$ vertices with at most $m$ edges such that $\hom(H,T) = \Omega_H(m^{|H| - \alpha^*(H)} n^{2\alpha^*(H) - |H|})$.
\end{lemma}

\begin{proof}
We will let $G$ be a threshold graph with $3$ parts:
    \[\underbrace{11 \hdots 1}_{\alpha \text{ 1's}} \underbrace{00 \hdots 0}_{\beta \text{ 0's}} \underbrace{11 \hdots 1}_{\gamma \text{ 1's}}, \]
\noindent where we let $\alpha = \lfloor \sqrt{m} \rfloor$, $\gamma = \lfloor m/(2n) \rfloor$, $\beta = n - \alpha - \gamma$. Then we have that there are at most
\[\frac{(\sqrt{m})^2}{2} + \frac{m}{2n}n = m\]
\noindent edges as desired.

To analyze the number of homomorphisms, we note that we lost at most a factor of two from the floors. That is $\alpha \geq \sqrt{m}/2$ and $\gamma \geq (m/4n)$. Clearly, we also have that $\beta \geq n - \sqrt{m} - m/(2n) \geq n/25$. 

Now let $f: V(H) \rightarrow \{0,1/2,1\}$ be an optimal fractional independence function i.e. $f(v_1), ..., f(v_n)$ is an optimal solution to the linear program used to define to the fractional independence number (see Definition \ref{def:fractional-independence-number}). Note that such a function exists since the feasibility polytope for fractional independence is half-integral. Now, we claim that any injective function $\varphi$ that sends $f^{-1}(1/2)$ to the first block of $1$'s, $f^{-1}(0)$ to the second block of $1$'s, and $f^{-1}(1)$ to the block of $0's$ is a homomorphism. This claim will complete the proof since there are at least 

\[\left( \frac{\sqrt{m}}{2} \right)^{|f^{-1}(1/2)|} \left(\frac{n}{25} \right)^{|f^{-1}(1)|} \left(\frac{m}{4n} \right)^{|f^{-1}(0)|} = \Omega_H(m^{|H| - \alpha^*(H)} n^{2\alpha^*(H) - |H|})\]
\noindent such functions. To see that these are homomorphisms, suppose $uv \in E(H)$. Without loss of generality assume $f(u) \leq f(v)$. Now if $u \in f^{-1}(0)$, then since every vertex in the final block of $1$'s is domininating we have that $\varphi(u) \varphi(v) \in E(G)$. Otherwise if $f(u) = f(v) = \frac{1}{2}$, then we have that both $\varphi(u)$ and $\varphi(v)$ are mapped to vertices in the first block of ones. Since any two vertices in this block are connected, we have $\varphi(u) \varphi(v) \in E(G)$ and this is indeed a homomorphism.
\end{proof}

\begin{corollary}
\label{cor:janson-lower-bound}
Let $H$ be a fixed graph and $c \in [0,1]$, then $\mathcal{M}_H(c) = \Omega_H(c^{|H| - \alpha^*(H)})$.
\end{corollary}

We will now show that there are graphs $H$ such that with only two parts we get far fewer homomorphisms as $c \rightarrow 0$.

\begin{proof}[Proof of Theorem \ref{threePartsHoms}]
Since isolated vertices do not affect the optimizing graph, we'll assume without loss of generality that $H$ has no isolated vertices.

By Corollary \ref{cor:janson-lower-bound}, we have that there exists a graph $G$ and constant $C_1$ such that
    \[t(H,G) \geq C_1 c^{|H|-\alpha^*(H)}.\]
This implies that $\mathcal{M}_H(c) \geq C_1 c^{|H|-\alpha^*(H)}$ by Lemma \ref{homSeqLemma}. For the sake of contradiction, assume that $H$ is optimized on the quasi-star or quasi-clique. Note that if it's optimized on the clique then we have that for all graphs $G$
    \[t(H,G) \leq c^{|H|/2}.\]

\noindent Since $\alpha^*(H) > |H|/2$, we have that for $c$ sufficiently small $C_1 c^{|H|-\alpha^*(H)} > c^{|H|/2}$, a contradiction.

So we must have that the homomorphism density of $H$ is maximized on the quasi-star. Note that in any homomorphism from $H$ to the quasi-star, the vertices in $H$ mapping to $0$ vertices in the quasi-star must form an independent set. Hence we have at least $|H| - \alpha(H)$ vertices are mapped to one of the at most $cn$ dominating vertices in the quasi-star of edge density $c$. By a union bound argument, this implies that there are at most $O_H(c^{n-\alpha(H)} n^{|H|})$ homomorphisms. But now note again that for $c$ sufficiently small we have that $C_2 c^{|H|-\alpha(H)} < C_1 c^{|H|-\alpha^*(H)}$, which is a contradiction.

Thus such a graph $H$ is optimized on neither the quasi-star nor the quasi-clique.
\end{proof}

\begin{corollary}
There exist graphs that are optimized on neither the quasi-star nor the quasi-clique. Moreover these graphs can be taken to be connected and threshold.
\end{corollary}
\begin{proof}
Take $H = K_3 \sqcup K_{1,2}$. Then $\alpha(H) = 3$ and $\alpha^*(H) = 3.5$. The result then follows by Theorem \ref{threePartsTheorem}.

If we wish to take the graph to be connected, adding edges from each vertex in the clique to the vertex of degree $2$ in the star gives a graph $H$ where we still have $\alpha(H) = 3$ and $\alpha^*(H) = 3.5$. Moreover, such a graph is a threshold graph. We can also take $H = K_\ell \sqcup K_{1,2}$ for any $\ell \geq 3$ if we want arbitrarily large examples.
\end{proof}

\subsection{Maximizing The Number of 2-Stars}
We now rederive results of Ahlswede and Katona \citep{ahlswede1978graphs} for maximizing the number of two-stars in a graph. 

\cherryMaximizationTheorem*

\begin{proof}
Let $c > 0$ and consider the limiting threshold graph $T$ with edge density at most $c$ that achieves $\mathcal{M}_H(c)$. Note that we will not treat limiting threshold graphs here, and refer the reader to \cite{diaconis2008threshold} for more details. Now let $T'$ denote the threshold graph with edge density at most $c$ that maximizes $t(H,T')$ among all limiting threshold graphs with at most some finite number of parts $f$. Moreover, let $T'$ have the minimum number of parts among all such graphs. We'll show that $T'$ only has $2$ parts. The claim will then follow from a result of \cite{diaconis2008threshold} which states that for any $\epsilon > 0$ there is a limiting threshold graph $T^*$ with finitely many parts and edge density at most $c$ such that $t(H,T^*) \geq t(H,T) - \epsilon$.

For the sake of contradiction, suppose $T'$ has $\ell > 2$ parts. 
Recall that we can write $T'$ as a sequence of blocks of $0$'s and $1$s as described in Section \ref{section:prelim}. Suppose that the sequence starts with a $0$ and the blocks have proportions $\alpha_1, \alpha_2, ..., \alpha_{\ell} \in [0,1]$. Denote by $k$ the quantity $\alpha_4 + \alpha_6 + ...$ i.e. the proportion of $1$s after $\alpha_3$. Moreover, in a minor abuse of notation we let $c$ denote the edge density in just blocks $\alpha_1, \alpha_2, \alpha_3$ and $d$ denote the constant $\alpha_1 + \alpha_2 + \alpha_3$.

Now note that we must have that $\alpha_1, \alpha_2, \alpha_3$ must optimize

\begin{equation}
\begin{aligned}
\max_{\alpha,\beta,\gamma} \quad &  \alpha (k + \beta) ^2 + \beta (\alpha + \beta + k)^2 + \gamma k^2 \\
\textrm{s.t.} \quad & 2\alpha \beta + \beta^2 = c \\
& \alpha + \beta + \gamma = d\\
  &\alpha, \beta, \gamma \geq 0 .   \\
\end{aligned}
\end{equation}

We can now solve for other variables in terms of $\beta$.
\[\alpha = \frac{c}{2\beta} - \frac{\beta}{2}\]
\[\gamma = d - \alpha - \beta = d - \beta - \left(\frac{c}{2\beta} - \frac{\beta}{2} \right) = d - \frac{\beta}{2} - \frac{c}{2\beta}. \]

Note we can then rewrite the objective as a function of $\beta$.

\[f(\beta) = \left( \frac{c}{2\beta} - \frac{\beta}{2} \right) (k + \beta) ^2 + \beta \left(\frac{c}{2\beta} - \frac{\beta}{2} + \beta + k\right)^2 + \left( d - \frac{\beta}{2} - \frac{c}{2\beta} \right) k^2\]

\[= -\frac{\beta^3}{4} + \beta c + \frac{c^2}{4 \beta} + k(2c + dk). \]

\noindent Now we can take derivatives and find
    \[f'(\beta) = \frac{4\beta^2c - 3\beta^4 - c^2}{4\beta^2},\]
and
    \[f''(\beta) = \frac{-3\beta^4+c^2}{2\beta^3}.\]

Now we note that $f''(\beta) \leq 0$ if and only if $\beta \geq \sqrt{c}/\sqrt[4]{3}$. Clearly a local maximum must have $f'(\beta) = 0$ and $f''(\beta) \leq 0$. However, we note that the positivity constraints on $\alpha$ imply that $\beta \leq \sqrt{c}$. We then have that if $f''(\beta) \leq 0$ and $\beta < \sqrt{c}$ then $f'(\beta) > f'(\sqrt{c}) = 0$. Hence the optima must occur at the end points, where one of $\alpha$ or $\gamma$ is $0$. This contradicts the fact that $T'$ has $\ell$ parts.

Now suppose that it starts with a $1$ and the parts have proportions $\alpha_1, \alpha_2, ..., \alpha_{\ell} \in [0,1]$. Denote by $k$ the quantity $\alpha_5 + \alpha_7 + ...$ i.e. the proportion of $1$s after $\alpha_3$. Again, in a minor abuse of notation we let $c$ denote the edge density in just blocks $\alpha_1, \alpha_2, \alpha_3$ and $d$ denote the constant $\alpha_1 + \alpha_2 + \alpha_3$. Then $\alpha_1, \alpha_2, \alpha_3$ optimize

\begin{equation}
\begin{aligned}
\max_{\alpha,\beta,\gamma} \quad &  \alpha (k + \alpha + \gamma )^2 + \beta (\gamma + k)^2 + \gamma (k + d)^2 \\
\textrm{s.t.} \quad & \alpha^2 + 2\alpha \gamma + \gamma^2 + 2\beta \gamma = c \\
& \alpha + \beta + \gamma = d\\
  &\alpha, \beta, \gamma \geq 0 .   \\
\end{aligned}
\end{equation}

We can rewrite the edge density constraint as 
    \[(\beta - d)^2 + 2\beta \gamma = c.\]
As before, we will turn the problem into a univariate one by solving in terms of $\beta$.
    \[\gamma = \frac{c - (\beta-d)^2}{2\beta},\]
\noindent and 
    \[\alpha = d - \beta - \frac{c - (\beta-d)^2}{2\beta}.\]
Substituting into the objective function gives
    \[f(\beta) = \left(d - \beta - \frac{c - (\beta-d)^2}{2\beta} \right) \left(k + d - \beta \right)^2 + \beta \left(\frac{c - (\beta-d)^2}{2\beta} + k \right)^2 + \left(\frac{c - (\beta-d)^2}{2\beta}\right) (k + d)^2. \]
    
Taking derivatives then gives us that
    \[f'(\beta) = \frac{-(\beta^2 + c - d^2)(3\beta^2+c-d^2)}{4\beta^2},\]
\noindent and
    \[f''(\beta) = \frac{-3\beta^4 + (c-d^2)^2}{2\beta^3}.\]
Now we note that if $f''(\beta) \leq 0$ then $\beta \geq \frac{\sqrt{d^2 - c}}{\sqrt[4]{3}} $. So if $f''(\beta) \leq 0$, then $3\beta^2+c-d^2 \geq 0$. We note that if there is equality, then $c = d^2$ and $\beta = 0$, contradicting the minimality of $T'$. So if $f'(\beta) = 0$ and $f''(\beta) \leq 0$, we must have that $\beta = \sqrt{d^2 - c}$. But then we have that $\alpha = 0$, which is again a contradiction.

Since we have reached a contradiction in all possible cases, we have that $\ell \leq 2$ and that the optimizing graph can be taken to be either the quasi-star or the quasi-clique.
\end{proof}

\section{Conclusion}
We end with a few open questions. First, perhaps the most natural question to ask

\begin{question}
\label{fewPartsQ}
For any $c \in [0,1]$ and graph $H$ is it true that $t(H, \cdot)$ asymptotically maximized on a threshold graph with finitely many parts? Is there a bound on the number of parts that is independent of $H$?
\end{question}

An intermediate question easier to resolve than the above, but still of interest would be

\begin{question}
Can the approach of Theorem \ref{cherryMaximizationTheorem} be generalized to work for $k$-stars and more general graphs? Can the one-dimensional graphons for threshold graphs help us solve the problem for various $H$?
\end{question}

We also remark that our result in Theorem \ref{sparseThresholdMaximizationTheorem} is likely not tight and can probably be extended to sparser graphs, leading us to the next question

\begin{question}
Does Theorem \ref{sparseThresholdMaximizationTheorem} hold whenever $m = \omega(n)$?
\end{question}

\noindent \textbf{Acknowledgements:} We thank Annie Raymond, Mohit Singh and Rekha Thomas for useful conversations. Grigoriy Blekherman was partially supported by NSF grant DMS-1901950. We thank the anonymous referee for their comments which helped to greatly improve the paper.

\nocite{*}
\bibliographystyle{acm}
\bibliography{bibliography}

\begin{thebibliography}{10}

\bibitem{ahlswede1978graphs}
{\sc Ahlswede, R., and Katona, G.~O.}
\newblock Graphs with maximal number of adjacent pairs of edges.
\newblock {\em Acta Mathematica Hungarica 32}, 1-2 (1978), 97--120.

\bibitem{alon1981number}
{\sc Alon, N.}
\newblock On the number of subgraphs of prescribed type of graphs with a given
  number of edges.
\newblock {\em Israel Journal of Mathematics 38}, 1-2 (1981), 116--130.

\bibitem{alon2004probabilistic}
{\sc Alon, N., and Spencer, J.~H.}
\newblock {\em The probabilistic method}.
\newblock John Wiley \& Sons, 2004.

\bibitem{day2019conjecture}
{\sc Day, A.~N., and Sarkar, A.}
\newblock On a conjecture of {N}agy on extremal densities.
\newblock {\em arXiv preprint arXiv:1910.13465\/} (2019).

\bibitem{diaconis2008threshold}
{\sc Diaconis, P., Holmes, S., and Janson, S.}
\newblock Threshold graph limits and random threshold graphs.
\newblock {\em Internet Mathematics 5}, 3 (2008), 267--320.

\bibitem{gerbner2018maximum}
{\sc Gerbner, D., Nagy, D.~T., Patk{\'o}s, B., and Vizer, M.}
\newblock On the maximum number of copies of {$H$} in graphs with given size
  and order.
\newblock {\em arXiv preprint arXiv:1810.00817\/} (2018).

\bibitem{janson2004upper}
{\sc Janson, S., Oleszkiewicz, K., and Ruci{\'n}ski, A.}
\newblock Upper tails for subgraph counts in random graphs.
\newblock {\em Israel Journal of Mathematics 142}, 1 (2004), 61--92.

\bibitem{kenyon2017multipodal}
{\sc Kenyon, R., Radin, C., Ren, K., and Sadun, L.}
\newblock Multipodal structure and phase transitions in large constrained
  graphs.
\newblock {\em Journal of Statistical Physics 168}, 2 (2017), 233--258.

\bibitem{kopparty2011homomorphism}
{\sc Kopparty, S., and Rossman, B.}
\newblock The homomorphism domination exponent.
\newblock {\em European Journal of Combinatorics 32}, 7 (2011), 1097--1114.

\bibitem{lovasz2012large}
{\sc Lov{\'a}sz, L.}
\newblock {\em Large networks and graph limits}, vol.~60.
\newblock American Mathematical Soc., 2012.

\bibitem{mahadev1995threshold}
{\sc Mahadev, N.~V., and Peled, U.~N.}
\newblock {\em Threshold graphs and related topics}.
\newblock Elsevier, 1995.

\bibitem{nagy2017number}
{\sc Nagy, D.~T.}
\newblock On the number of $4$-edge paths in graphs with given edge density.
\newblock {\em Combinatorics, Probability and Computing 26}, 3 (2017),
  431--447.

\bibitem{reiher2018maximum}
{\sc Reiher, C., and Wagner, S.}
\newblock Maximum star densities.
\newblock {\em Studia Scientiarum Mathematicarum Hungarica 55}, 2 (2018),
  238--259.

\bibitem{reiterman1985threshold}
{\sc Reiterman, J., R{\"o}dl, V., {\v{S}}i{\v{n}}ajov{\'a}, E., and T\r{u}ma,
  M.}
\newblock Threshold hypergraphs.
\newblock {\em Discrete Mathematics 54}, 2 (1985), 193--200.

\bibitem{sidorenko1994partially}
{\sc Sidorenko, A.}
\newblock A partially ordered set of functionals corresponding to graphs.
\newblock {\em Discrete Mathematics 131}, 1-3 (1994), 263--277.

\end{thebibliography}

\end{document}